\documentclass[reqno,10pt]{amsart}

\usepackage[english]{babel}
\usepackage{dsfont}
\usepackage[bookmarksnumbered,colorlinks]{hyperref}
\usepackage{enumerate}
\usepackage{graphicx}
\usepackage{xcolor}
\newcommand{\df}{{\rm d}}
\newcommand{\R}{\mathds R}
\newcommand{\N}{\mathds N}

  \def\br#1\er{\textcolor{red}{#1}} %

\title[Almost isometries of non-reversible metrics]{Almost isometries of non-reversible metrics with applications to stationary spacetimes}

\author[M. A. Javaloyes]{Miguel Angel Javaloyes}
\address{Departamento de Matem\'aticas, \hfill\break\indent
Universidad de Murcia, \hfill\break\indent
Campus de Espinardo,\hfill\break\indent
30100 Espinardo, Murcia, Spain}
\email{majava@um.es}

\author[L. Lichtenfelz]{Leandro Lichtenfelz}
\address{Departamento de Matem\'atica,\hfill\break\indent
Universidade de S\~ao Paulo, \hfill\break\indent Rua do Mat\~ao
1010,\hfill\break\indent CEP 05508-900, S\~ao Paulo, SP, Brazil}
\email{silverratio@gmail.com}

\author[P. Piccione]{Paolo Piccione}
\address{Departamento de Matem\'atica,\hfill\break\indent
Universidade de S\~ao Paulo, \hfill\break\indent Rua do Mat\~ao
1010,\hfill\break\indent CEP 05508-900, S\~ao Paulo, SP, Brazil}
\email{piccione.p@gmail.com}

\date{17.05.2012}

\thanks{MAJ was partially supported by MINECO (Ministerio de Econom\'{\i}a y Competitividad) project MTM2012-34037, Regional J. Andaluc\'{\i}a Grant P09-FQM-4496 and Fundaci\'{o}n S\'{e}neca project 04540/GERM/06, Spain. This research is a
result of the activity developed within the framework of the Programme in Support of Excellence Groups of the Regi\'{o}n de Murcia, Spain, by Fundaci\'{o}n S\'{e}neca, Regional Agency for Science and Technology (Regional Plan for
Science and Technology 2007-2010).
   LL is sponsored by Fapesp, Brazil. PP is partially sponsored by CNPq and Fapesp, Brazil.}



  \def\br#1\er{\textcolor{red}{#1}} %

\begin{document}

\newtheorem{thm}{Theorem}[section]
\newtheorem{prop}[thm]{Proposition}
\newtheorem{lemma}[thm]{Lemma}
\newtheorem{cor}[thm]{Corollary}
\theoremstyle{definition}
\newtheorem{defi}[thm]{Definition}
\newtheorem{notation}[thm]{Notation}
\newtheorem{exe}[thm]{Example}
\newtheorem{conj}[thm]{Conjecture}
\newtheorem{prob}[thm]{Problem}
\newtheorem{rem}[thm]{Remark}

\begin{abstract}
We develop the basics of a theory of almost isometries for spaces endowed with a quasi-metric. The case of non-reversible Finsler (more specifically, Randers) metrics is of particular
interest, and it is studied in more detail. The main motivation arises from General Relativity,
and more specifically in spacetimes endowed with a timelike conformal field $K$,  in which case \emph{conformal
diffeomorphisms} correspond to almost isometries of the Fermat metrics defined in the spatial part.
A series of results on the topology and the Lie group structure of conformal maps are discussed.
\end{abstract}

\maketitle
\section{Introduction}
A quasi-metric is just a metric (in the context of metric spaces) without the restriction of symmetry.
In the last years, an increasing interest in quasi-metrics has risen  \cite{BhDa11,CRZ09,ColZim,KPQ10,MaVa10,Men11,Pl09,SG10}.
This interest can be justified by the amount of applications, since non-reversibility is present in many situations.

The first systematic study of quasi-metrics was carried out by W. A. Wilson in \cite{Wil31}. Later, also H. Busemann developed some results on
 metrics without the restriction of symmetry \cite{Bus44}, but with the additional condition that forward and backward balls generate the same topology (see Remark~\ref{generalized}).
Most of the results by H. Busemann were developed for (symmetric) metric spaces and it was E. M. Zaustinsky who extended some of them to quasi-metric spaces \cite{Zau}.

Our main goal is to study the automorphisms of a quasi-metric space $(X,d)$ that preserve what we call the triangular function, namely, the quantity $T(x,y,z)=d(x,y)+d(y,z)-d(x,z)$. This quantity measures how far the points $x,y,z$ are from achieving the equality in the triangle inequality. Maps that preserve the triangular function will be
called \emph{almost isometries} and it is an immediate consequence of the very definition that such maps preserve minimizing geodesics (see Corollary~\ref{pregeo}). It is easy to see that when symmetry holds, almost isometries are in fact isometries. Therefore,
this notion is relevant only in the non-symmetric context.

One of the fundamental examples of quasi-metrics is given by the distance associated to a non-reversible Finsler metric. As in the case of isometries, the almost isometries can be described in terms of the pullback of the Finsler metric. More precisely, they are the maps that preserve the Finsler metric up to an exact one-form, namely, the pullback of  a Finsler metric is the sum of the same Finsler metric and the differential of a smooth function. One of the simplest examples of non-reversible Finsler metrics are
Randers metrics, which are defined as the sum of the square root of a Riemannian metric and a one-form having norm less than one at every point. Recently, a relation between Randers metrics and standard stationary spacetimes has been explored
 \cite{BiJa11,CGS11,CJM10,CJM10b,CJM11,CJS11,DPS11,FHS10b,GGP09,GHWW09}. The core of this relation lies in the fact that, as a consequence of the Fermat principle, lightlike geodesics project up to parametrization into geodesics of a Randers metric that we call Fermat metric. Then, studying existence and multiplicity of lightlike geodesics between an event and a vertical line is equivalent to studying existence and multiplicity of geodesics of the Fermat metric \cite{BiJa11,CJM10,CJM10b,CJM11}. Moreover, causality of the stationary spacetime can be characterized in terms of completeness properties of the Fermat metric \cite{CJS11} and its causal boundary can be described in terms of the topological boundary of the Fermat metric \cite{FHS10b}. Giving continuity to the fruitful interplay between Randers metrics and stationary spacetimes, we will explore this relation in the level of transformation groups. In fact, we will relate the conformal maps of the conformastationary spacetime that preserve the timelike Killing vector field with the almost isometries of the Fermat metric obtaining, as a consequence of this interplay, results of genericity and compactness for the $K$-conformal group. This relation was the departing point and the inspiration to get to the concept of almost isometry.

Let us describe in detail the results of this paper. In Section~\ref{sec:basicsgendist}, we introduce the notions of triangular function of a quasi-metric and almost isometry (Definition~\ref{almostiso}), the latter being a map that preserves the triangular function. Then we show in Proposition~\ref{defequiv} that the definition of almost isometry $\varphi$ of $(X,d)$ is equivalent to the existence of a real function $f:X\rightarrow \R$ such that
\[d(\varphi(x),\varphi(y))=d(x,y)+f\big(\varphi(y)\big)-f\big(\varphi(x)\big)\]
for every $x,y\in X$. Moreover, we show that $\varphi$ is always a homeomorphism and $f$ is continuous (Lemma \ref{lem:homeo}). Finally we show that the extended isometry group $\widetilde{\rm Iso}(X,d)$, made up by the almost isometries of $(X,d)$, is contained in the isometry group of the symmetrized metric in \eqref{distances} and it is a topological group (Proposition~\ref{fisometriesLie}). In Subsection~\ref{subsec:localisom}, we study local almost isometries and we conclude in Theorem~\ref{thm:propsimpcon} that a local almost isometry between two length spaces with weakly finitely compact domain and simply connected codomain must be a global almost isometry. A counterexample to this result in the case that the quasi-metric spaces are not length spaces
is provided in Remark~\ref{counterex}.

In Section~\ref{finslermetrics} we prove that an almost isometry of a Finsler manifold is an isometry of the symmetrized Finsler metric in \eqref{symmetrized} and then it is smooth (Lemma \ref{thm:diffalmostisomFinsler}). Moreover, the function $f:M\rightarrow\R$ of Proposition~\ref{defequiv} is smooth and $\varphi$ is an almost isometry for $F$ if and only if $\varphi_*(F)=F-{\rm d}f$ (Proposition~\ref{diff_fisometries}).  Finally, the extended isometry group
  of $(S,F)$, denoted by $\widetilde{\rm Iso}(S,F)$, is a closed subgroup of the isometry group of the symmetrized Finsler metric $\hat{F}$, which is a Lie group (see for instance \cite{DengHou02}) and then $\widetilde{\rm Iso}(S,F)$ is also a Lie group (Proposition~\ref{isoextended}).

 In Section \ref{stationary} we first introduce $K$-conformastationary decompositions $(M=S\times\R,g^K)$ of conformastationary spacetimes endowed  with a complete timelike conformal field $K$  in \eqref{e1} and the Fermat metric $F^K$ associated to them \eqref{fermatmetric}. In Theorem~\ref{fundamental} we show that a conformal map $\psi:(M,g)\rightarrow (M,g)$  determines an almost isometry $\varphi:(S,F^K)\rightarrow (S,F^W)$ of the Fermat metrics $F^K$ and $F^W$ associated to one of the conformastationary decompositions determined respectively by $K$ and $W=\psi_*(K)$. In particular $\varphi$ is given by the spatial component of the conformal map  $\psi$. Then we define a $K$-conformal map as a map that is conformal and preserve the timelike conformal vector field $K$.  In Proposition \ref{bijection} we show that an almost isometry determines a $K$-conformal map up to a composition with an element of the closed subgroup generated by $K$, which will be denoted by ${\mathcal K}$. Indeed, there is a Lie group homomorphism between the $K$-conformal maps quotiented by ${\mathcal K}$ and the extended isometry group. Furthermore, in Corollaries~\ref{genericity} and \ref{compactLiegroup}, we use the former Lie group homomorphism to obtain a genericity result for standard stationary spacetimes with discrete $K$-conformal group and the compactness of the $K$-conformal group.

In the last subsection of Section~\ref{stationary}, we obtain some consequences for the conformal group ${\rm Conf}(M,g)$. In particular in Corollary~\ref{compactconf} we give a characterization of the compactness of ${\rm Conf}(M,g)/{\mathcal K}$. Finally in Theorem~\ref{finalTh} we collect the one-to-one relation between conformal maps and almost isometries of the Fermat metrics up to composition with elements of ${\mathcal K}$.

\section{Quasi-metrics and almost isometries}\label{sec:basicsgendist}
Let us first of all introduce the concept of a quasi-metric (see \cite{Wil31}).
\begin{defi}\label{symmetricdist}
Given a set $X$, we say that a function $d:X\times X\rightarrow \R$ is a {\em quasi-metric} if
\begin{enumerate}
\item[(i)] $d(x,y)\geq 0$ for every $x,y\in X$ and $d(x,y)=0$ if and only if $x=y$,
\item[(ii)] $d(x,y)+d(y,z)\geq d(x,z)$ (triangle inequality).
\end{enumerate}
\end{defi}
As a consequence of the lack of symmetry, there are two kinds of balls, namely, \emph{forward and backward balls},
defined by $B_d^+(x,r)=\{y\in X: d(x,y)<r\}$ and
$B_d^-(x,r)=\{y\in X: d(y,x)<r\}$ respectively, for $x\in X$ and $r>0$. Both families generate two topologies that we will call respectively forward and backward topologies. 

A pair $(X,d)$, where $d$ is a quasi-metric on the set $X$, will be called a \emph{quasi-metric space} 
and it will always be assumed to be endowed with the topology induced by the family
$B_d^+(x,r)\cap B_d^-(x,r)$, $x\in M$ and $r>0$, which is finer than the forward and the backward topologies. Let us observe that this topology coincides with the topology generated by (the balls of) the symmetrized metric
\begin{equation}\label{distances}
\widetilde{d}(x,y)=\tfrac 12\big(d(x,y)+d(y,x)\big).
\end{equation}
In fact, given $\{x_n\}_{n\in\N}$ in $X$, $x_n\rightarrow x$ in $(X,d)$ if and only if $\lim_{n\rightarrow\infty}d(x_n,x)=\lim_{n\rightarrow\infty}d(x,x_n)=0$.  For that reason we will refer to this topology as ``the symmetric topology'' associated to the
quasi-metric.

\begin{rem}\label{generalized}
Let us observe that the notion of general metric space given in \cite[Section 1]{Bus44} and \cite[page 5]{Zau}) is a particular case of quasi-metric space. Indeed, a general metric space is a quasi-metric space satisfying the condition
\begin{itemize}
\item[(GMS)]{\it  Given a sequence $\{x_n\}_{n\in\N}\subset X$ and $x\in X$, then $\lim_{n\rightarrow \infty}d(x_n,x)=0$ if
and only if $\lim_{n\rightarrow \infty}d(x,x_n)=0$.}
\end{itemize}
This condition is equivalent to the fact that the forward and backward topologies coincide with the topology of $(X,d)$. Moreover, following \cite{Men04}, the terminology ``quasi-metric space'' is usually used for a space having a quasi-metric, but endowed with the forward topology,  whereas ``asymmetric metric space'' is used for
a quasi-metric space endowed with the ``symmetric topology''. Let us point out that we will also assume that quasi-metric space be endowed with the symmetric topology  but we will prefer to use ``quasi-metric space'' rather than ``asymmetric metric space''.
\end{rem}

In a quasi-metric space we can define the length of a continuous curve
$\alpha:[a,b]\subseteq \R\rightarrow X$ as
\begin{equation}\label{d-length}\ell(\alpha)=\sup_{\mathcal P}\sum_{1=1}^r d(\alpha(s_i),\alpha(s_{i+1})),
\end{equation}
where $\mathcal P$ is the set of partitions $a=s_1<s_2<\ldots<s_{r+1}=b$, $r\in \N$. We say that $\alpha$ is {\it rectifiable} when $\ell(\alpha)$ is finite. Moreover, we say that a curve $\gamma$ in $X$
from $p$ to $q$ is a {\it minimizing geodesic} if $\ell(\gamma)=d(p,q)$. Let us define the
{\it triangular function} $T:X\times X\times X\rightarrow\left[0,+\infty\right[$ of a quasi-metric space $(X,d)$ as
$T(x,y,z)=d(x,y)+d(y,z)-d(x,z)$ for every $x,y,z\in X$. Evidently, $T$ is continuous.
\begin{prop}\label{CharGeo}
A curve $\alpha:[a,b]\subseteq\R\rightarrow X$ is a minimizing geodesic of
a quasi-metric space $(X,d)$ if and only if $T(\alpha(s_1),\alpha(s_2),\alpha(s_3))=0$ for every $a\leq s_1<s_2<s_3\leq b$.
\end{prop}
\begin{proof}
First observe that the triangle inequality implies that if $\alpha$ is a minimizing geodesic, any restriction $\alpha|_{[\tilde{a},\tilde{b}]}$ is also a minimizing geodesic.
With this  observation, one proves easily that for a minimal geodesic
$\alpha:[a,b]\to X$, the quantity $T(\alpha(s_1),\alpha(s_2),\alpha(s_3))$ vanishes
for every $a\leq s_1<s_2<s_3\leq b$.

The converse follows from the definition of minimizing geodesic, since, in this case, we get easily that
\[\sum_{1=1}^r d(\alpha(s_i),\alpha(s_{i+1}))=d(\alpha(a),\alpha(b))\]
for any partition $a=s_1<s_2<\ldots<b=s_{r+1}$, $r\in \N$.
\end{proof}
\begin{defi}
We say that a quasi-metric space $(X,d)$ is {\em finitely compact} if $B_d^+(x,r)\cup B_d^-(x,r)$ is precompact for every $x\in X$ and every $r>0$ and {\em weakly finitely compact} if $B_d^+(x,r)\cap B_d^-(x,r)$ is precompact for every $x\in X$ and every $r>0$.
\end{defi}
A sequence  $\{x_n\}_{n\in\N}$ in $X$  is said to be a Cauchy sequence of $(X,d)$ if for every $\epsilon>0$ there exists $N\in\N$ such that $d(x_n,x_m)<\epsilon$ for every $n,m\ge N$. Moreover, we say that
it is a forward (resp. backward) Cauchy sequence of $(X,d)$ if for every
$\epsilon>0$ there exists $N\in\N$ such that $d(x_n,x_m)<\epsilon$ for every $N\leq n<m$ (resp. $N\leq m<n$). We say that a quasi-metric space is (forward, backward) complete if all the (forward, backward) Cauchy sequences converge.  We will say that
a quasi-metric space is complete when all the Cauchy sequences converge. Let us point out that this definition is different from the one usually used in Finsler metrics, where complete means forward and backward complete. With our definition, complete is weaker than forward or backward complete.   Let us recall in the following lemma several facts regarding completeness.
\begin{lemma}\label{impli}
Consider a quasi-metric space $(X,d)$:
\begin{enumerate}[(i)]
\item If $(X,d)$ is finitely compact, then it is weakly finitely compact.
\item If $(X,d)$ is finitely compact, then it is forward and backward complete.
\item If $(X,d)$ is weakly finitely compact, then it is complete.
\item $(X,d)$ is complete if and only if $(X,\tilde{d})$ (see \eqref{distances}) is complete.
\end{enumerate}
\end{lemma}

Given quasi-metric spaces $(X_i,d_i)$, $i=1,2$, a bijection $\varphi:X_1\to X_2$  is said to be an \emph{isometry} if $d_2\big(\varphi(x),\varphi(y)\big)=d_1(x,y)$
for all $x,y\in X_1$.
Evidently, an isometry is a homeomorphism when $X_i$ is endowed
with the topology induced by $d_i$, $i=1,2$.
Given a quasi-metric space $(X,d)$, the group of isometries $\varphi:X\to X$ will be denoted by $\mathrm{Iso}(X,d)$;
this group will be endowed with the compact-open topology, and we will show below that such topology makes $\mathrm{Iso}(X,d)$
a topological group.

We will be interested in a more general class of transformations of quasi-metric spaces
defined as follows.
\begin{defi}\label{almostiso}
Let $(X_1,d_1)$ and $(X_2,d_2)$ be two quasi-metric spaces.
A bijection $\varphi:X_1\rightarrow X_2$ is an {\em almost isometry}
if it preserves the triangular function, that is,
\[T_2(\varphi(x),\varphi(y),\varphi(z))=T_1(x,y,z)\]
for every $x,y,z\in X_1$, where $T_1$ and $T_2$ are the triangular functions associated respectively to $(X_1,d_1)$ and $(X_2,d_2)$.
\end{defi}
\begin{cor}\label{pregeo}
Almost isometries preserve minimizing geodesics.
\end{cor}
\begin{proof}
A straightforward consequence of Proposition \ref{CharGeo}.
\end{proof}
\begin{prop}\label{defequiv}
Given quasi-metric spaces $(X_1,d_1)$ and $(X_2,d_2)$, a bijection $\varphi:X_1\rightarrow X_2$ is an almost isometry if and only
if there
exists a real function $f$ on $X_2$ such that
\begin{equation}\label{eq:defquasiisom}
d_2\big(\varphi(x),\varphi(y)\big)=d_1(x,y)+f\big(\varphi(y)\big)-f\big(\varphi(x)\big)
\end{equation}
for every $x,y\in X_1$.
\end{prop}
\begin{proof}
Assume that $\varphi:X_1\rightarrow X_2$ is an almost isometry. Fix a point $x_0\in X_1$ and define  $f:X_2\rightarrow \R$ as $f(z)=d_1(\varphi^ {-1}(z),x_0)-d_2(z,\varphi(x_0))$ for every $z\in X_2$. Given $x,y\in X_1$, as $\varphi$ preserves the triangular function, we have
\[d_1(x,y)+d_1(y,x_0)-d_1(x,x_0)=d_2(\varphi(x),\varphi(y))+d_2(\varphi(y),\varphi(x_0))-d_2(\varphi(x),\varphi(x_0)),\]
which is equivalent to \eqref{eq:defquasiisom}. The converse is a straightforward computation.
\end{proof}
Observe that if the quasi-metrics $d_1$ and $d_2$ are symmetric, then
the function $f$ has to be constant and $\varphi$ must be an isometry.
\begin{cor}
Given metric spaces $(X_1,d_1)$ and $(X_2,d_2)$, a bijection $\varphi:X_1\rightarrow X_2$ is an isometry if and only if it preserves the
triangular function.
\end{cor}
Let us remark that the function $f$ of Proposition \ref{defequiv} is determined up to a constant, that is, if $f$ satisfies \eqref{eq:defquasiisom}, all the functions satisfying
\eqref{eq:defquasiisom} are the functions $f+c$ with $c\in\R$.
\begin{lemma}\label{lem:homeo}
If $\varphi:X_1\rightarrow X_2$ is an almost isometry, then $\varphi$ is a homeomorphism
when $X_1$ and $X_2$ are endowed with the topologies induced by $d_1$ and $d_2$ respectively.
Moreover, any function $f:X_2\rightarrow\mathds R$ for which \eqref{eq:defquasiisom} holds is continuous.
\end{lemma}
\begin{proof}
An almost isometry $\varphi:X_1\to X_2$ is an isometry of  (symmetric)  metric spaces,
when $X_1$ and $X_2$ are endowed with the symmetrized quasi-metrics $\widetilde d_1$ and $\widetilde d_2$ respectively (see \eqref{distances} and Proposition \ref{defequiv}). Thus, $\varphi$ is a homeomorphism
in the topologies induced by $\widetilde d_1$ and $\widetilde d_2$, which
coincide
with the topologies induced by the quasi-metrics $d_1$ and $d_2$ (see again \eqref{distances}).
This proves the first statement. Since $\varphi$ is a homeomorphism,
the continuity of $f$ follows now from the continuity of the quasi-metrics $d_1$ and $d_2$.
\end{proof}
Let us observe that composition provides a structure of group for the subset of almost isometries. Given quasi-metric
spaces $(X_i,d_i)$, $i=1,2,3$, and two almost isometries $\varphi_1:X_1\to X_2$ and  $\varphi_2:X_2\to X_3$,  the composition $\varphi_2\circ\varphi_1:X_1\to X_3$ is also an almost isometry.
Moreover, the inverse of an almost isometry is also an almost isometry.
In particular, the set of almost isometries $\varphi:X\to X$ of a quasi-metric space $(X,d)$,
denoted by $\widetilde{\mathrm{Iso}}(X,d)$, is a group with the operation of composition;
it will be called the \emph{extended isometry group} of $(X,d)$.
\begin{prop}\label{fisometriesLie}
With the above notation, $\widetilde{\mathrm{Iso}}(X,d)$ and ${\rm Iso}(X,d)$ are topological groups endowed with the compact-open topology.
If the topology induced by $d$ is locally compact, then $\widetilde{\mathrm{Iso}}(X,d)$ and ${\rm Iso}(X,d)$ are locally compact.
\end{prop}
\begin{proof}
Consider the metric $\widetilde{d}$ on $X$ defined
in \eqref{distances}. As the topology induced by this metric coincides with the one induced by $d$,
in particular, $(X,d)$ is locally compact
if and only if $(X,\widetilde d)$ is locally compact. Furthermore, an almost isometry for
$d$ is an isometry for $\widetilde{d}$. 
Hence
\begin{equation}\label{eq:interisos}
{\rm Iso}(X,d)\subseteq\widetilde{\rm Iso}(X,d)\subseteq{\rm Iso}(X,\widetilde{d}).
\end{equation}
Applying a classical result  (see for instance \cite[Proposition 5.2.1]{Pestov06}) to the metric space $(X,\widetilde{d})$,
we deduce that ${\rm Iso}(X,\widetilde{d})$  endowed with the compact-open topology is a topological group, locally compact when the topology of
$(X,\widetilde d)$  is locally compact. Then, by the inclusion \eqref{eq:interisos}, ${\rm Iso}(X,d)$ and  $\widetilde{\rm Iso}(X,d)$  are also  topological groups with the compact-open topology.
Note that ${\rm Iso}(X,d)$ and  $\widetilde{\rm Iso}(X,d)$  are closed with respect to the compact-open topology\footnote{In fact, pointwise limits of almost isometries is an almost isometry by the continuity of the triangular function.},
which implies that they inherit the local compactness from ${\rm Iso}(X,\widetilde{d})$.
\end{proof}

\smallskip
\subsection{Local almost isometries}\label{subsec:localisom} The goal of this section is to find conditions implying that a map that is locally an almost isometry is in fact a global almost isometry.
\begin{defi}
Let $(X_1, d_1)$ and $(X_2, d_2)$ be two quasi-metric spaces. We say that a map $\varphi : X_1 \rightarrow X_2$ is a \emph{local almost isometry} if for every $x \in X_1$, there exist $U \subseteq X_1$, $V \subseteq X_2$ open subsets, with $x\in U$, such that $\varphi|_U:U\rightarrow V$ is an almost isometry.
\end{defi}
Let us fix a local almost isometry $\varphi:X_1\to X_2$ throughout, and let us introduce some terminology. Assume that $V$ is an
open subset of $X_2$, that $U$ is an open subset of $X_1$ which is mapped homeomorphically onto $V$ by $\varphi$,
and that $f:V\to\mathds R$ is a function such that \eqref{eq:defquasiisom} holds for all $x,y\in U$ that are
\emph{sufficiently close to each other}.\footnote{\label{foo:nottransitive}This means that the set $\mathcal A_U=\big\{(x,y)\in U\times U:\text{\eqref{eq:defquasiisom} holds}\big\}$
contains a neighborhood of the diagonal of $U\times U$. Observe that this does \emph{not} imply that \eqref{eq:defquasiisom} holds for every $x,y\in U$,
because the subset $\mathcal A_U\subset U\times U$ is not an equivalence relation in $U$ (it is not transitive).}
We then say that $(f,U,V)$ is a \emph{$\varphi$-system}, or shortly, that $f$ is a \emph{$\varphi$-map} with domain $V$ (the role played by $U$ is not so relevant).
\smallskip

We list a few important properties of $\varphi$-systems:
\begin{itemize}
\item[(P1)] If $f$ is a $\varphi$-map, then $f+c$ is a $\varphi$-map (with the same domain as $f$) for all constant $c\in\mathds R$.
\item[(P2)] If $f$ is a $\varphi$-map with domain $V$, and $V'\subset V$ is open, then $f\vert_{V'}$ is a $\varphi$-map with domain $V'$.
\item[(P3)] Given two $\varphi$-maps $f_i:V_i\to\mathds R$, $i=1,2$,  if $p\in V_1\cap V_2$ and $f_1(p)=f_2(p)$, then $f_1$ and $f_2$
coincide in the connected component of $V_1\cap V_2$ that contains $p$.
\item[(P4)] The property of being a $\varphi$-map is \emph{local}, i.e., if $V$ is covered by a family of open subsets $V_\alpha$, $\alpha\in A$,
and $f:V\to\mathds R$ is a function such that the restriction $f\vert_{V_\alpha}$ is a $\varphi$-map with domain $V_\alpha$ for all $\alpha\in A$,
then $f$ is a $\varphi$-map with domain $V$.
\end{itemize}
Note that property (P4) would not hold if we required that equality \eqref{eq:defquasiisom}
should hold for all $x$ and $y$ in the domain of a $\varphi$-map, see footnote~\ref{foo:nottransitive}.
\begin{lemma}\label{thm:globalphifunct}
Assume that $\varphi:X_1\to X_2$ is a homeomorphism which is a local almost isometry and
 $X_2$ (or equivalently $X_1$) is connected, locally arc-connected and simply connected.
Then there exists a $\varphi$-map $f$ whose  domain is the whole $X_2$.
\end{lemma}
\begin{proof}
This is a special case of a globalization theory for sheaves,\footnote{The central idea in the theory developed in \cite[Appendix~B]{Gstructbook}
 is to look at the sheaf of \emph{germs} of $\varphi$-maps, and show that the projection onto $X_2$ is a covering map.
 A substantial simplification in the proof discussed in \cite{Gstructbook}
occurs here because, by property (P3), the germ of a $\varphi$-map $f$ at some point $p$ can be identified with the value of $f$ at $p$, and, by property (P1),
the sheaf of germs of $\varphi$-maps, as a set, is identified with the product $X_2\times \mathds R$ (the pair $(x,c)\in X_2\times\mathds R$ is identified with the unique germ
of the $\varphi$-map $f$ whose value at $x$ is $c$).} which can be found, for instance, in \cite[Appendix~B]{Gstructbook};
for the reader's convenience, we will provide a short proof here.

Consider in the set $X_2\times\mathds R$ the topology that has as a basis of open sets the graphs of $\varphi$-maps. This is \emph{not} the product topology,
and it is locally arc-connected, because so is every open subset of $X_2$.
With this topology, the projection $\pi:X_2\times\mathds R\to X_2$ is a covering map.  Namely, any open set $V$ of $X_2$ which is the domain of a $\varphi$-map
$f:V\to\mathds R$ is a fundamental open set; in fact, the inverse image $\pi^{-1}(V)$ is the disjoint union of the graphs of the $\varphi$-maps
$f+c$, with $c\in\mathds R$ constant (by properties (P1) and (P3)), and each one of such open subsets of $X_2\times\mathds R$ is mapped homeomorphically onto $V$ (by property (P2)).
Moreover, the set of domains of $\varphi$-maps is an open covering of $X_2$, because $\varphi$ is a local almost isometry.

Since $X_2$ is simply connected, then the restriction of $\pi$ to an arc-connected component of $X_2\times\mathds R$, which is open
because $X_2\times \mathds R$ is locally arc-connected, is a homeomorphism onto
$X_2$. The inverse of such a homeomorphism is the graph of a function $f:X_2\to\mathds R$ which is locally a $\varphi$-map; hence, by property (P4),
$f$ is a $\varphi$-map with domain $X_2$. This concludes the proof.
\end{proof}
Recall that given $(X,d)$, we can define an associated distance (which is not always a metric) $d_l$ given by the infimum of the lengths of curves between two points (the length computed as in \eqref{d-length}). We say that $(X,d)$ is a {\em length space} when $d_l=d$.
\begin{thm}\label{thm:propsimpcon}
 Let $\varphi:(X_1,d_1)\rightarrow (X_2,d_2)$ be a local almost isometry. Assume that $(X_1,d_1)$ and $(X_2,d_2)$ are length spaces, $d_1$ is weakly finitely compact and $X_2$ is locally arc-connected and simply connected. Then $\varphi$ is an almost isometry.
\end{thm}

\begin{proof}
Since $\varphi$ is a local almost isometry, $\varphi$ is a local isometry for the corresponding symmetrized quasi-metrics $\tilde{d}_1$ and $\tilde{d}_2$
(see \eqref{distances}). Observe that the weak finite compactness of $d_1$ implies that  $(X_1, \tilde{d}_1)$ is complete (see parts (iii) and (iv) of  Lemma \ref{impli}). Moreover, since $X_2$ is simply connected, it is not hard to show that $\varphi$ is a homeomorphism\footnote{%
Note that it is not proven, in general,  that a local isometry between complete length spaces is a covering map. Typically, in order to prove that a local isometry is a covering map
one needs an extra assumption on the local uniqueness and continuous dependence on the endpoints for geodesics, see for instance \cite[\S~3.4]{BuBuIv2001}.
Here, in order to prove the statement, one uses the fact that a local isometry from a complete metric space to another metric space has the \emph{unique lift property for paths}.
Moreover, any local homeomorphism from a path connected topological space to a simply connected topological space that has the unique lift property for paths is a homeomorphism.}.
 Let us denote by $\ell_1$ and $\ell_2$ the lengths associated to $(X_1,d_1)$ and $(X_2,d_2)$ respectively (see \eqref{d-length}) and let $f:X_2\rightarrow \R$ be the function obtained in Lemma \ref{thm:globalphifunct}. Given two arbitrary points $x,y\in X_1$ and a curve $\alpha:[a,b]\subseteq\R\rightarrow X_1$ from $x$ to $y$,  it follows from the definition of length in \eqref{d-length} and the fact that $\varphi$ is a local almost isometry that
\begin{equation*}
\ell_2(\varphi\circ\alpha)=\ell_1(\alpha)+f\big(\varphi(y)\big)-f\big(\varphi(x)\big).
\end{equation*}
As $(X_1,d_1)$ and $(X_2,d_2)$ are  length spaces, we deduce that
 \begin{equation*}
d_2(\varphi(x),\varphi(y))=d_1(x,y)+f\big(\varphi(y)\big)-f\big(\varphi(x)\big).
\end{equation*}
Finally, Proposition \ref{defequiv} concludes that $\varphi$ is an almost isometry.
\end{proof}
\begin{rem}\label{counterex}
Observe that when $(X_1,d_1)$ or $(X_2,d_2)$ are not length spaces, then a local almost isometry need not be an almost isometry. For example, if we
consider $X_1=X_2=\R$, $d_1$ the usual metric and
\[
d_2(x,y)=
\begin{cases}
|x-y|\quad\quad\quad\text{ if $|x-y|<1$,} \\

\quad\, 1 \quad\quad\quad\quad\quad\text{      otherwise}.
\end{cases}\]
Then the identity is a local isometry from $(X_1,d_1)$ to $(X_2,d_2)$ but not a global isometry.
\end{rem}

\section{Almost isometries and Finsler metrics}\label{finslermetrics}
A very important class of quasi-metrics comes from non-reversible Finsler metrics. In fact, we can define a more general class of metrics as follows.
Let $M$ be a manifold  and $\pi:TM\rightarrow M$  the natural projection from the tangent bundle to the manifold.
A {\em (continuous) pseudo-Finsler metric} is a continuous function $F:TM\rightarrow [0,+\infty)$ such that
\begin{enumerate}
\item it is fiberwise positively homogeneous of degree one, i.e., $F(\lambda v)=\lambda F(v)$ for every $v\in TM$ and $\lambda>0$
\item  $F(v)=0$ if and only if $v=0$.
\end{enumerate}
 We say that a pseudo-Finsler metric is smooth when
 $F$ is $C^\infty$ in $TM\setminus\mathbf 0$, i.e., it is smooth away from the zero section. In this case, we can define
the fundamental tensor $g_u$ as
\begin{equation}\label{fundamentaltensor}
g_u(v,w)=\left.\frac{\partial^2}{\partial s\partial t}\right|_{t=s=0}F^2(u+tv+sw),
\end{equation}
where $u\in TM\setminus \mathbf 0$ and $v,w\in T_{\pi(u)}M$. In the following
we will assume that a pseudo-Finsler metric is smooth unless we specify
explicitly that it is only continuous. Finally we define a {\em Finsler metric}
as a pseudo-Finsler metric with fiberwise strongly convex square, i.e., the fundamental tensor $g_u$ is positive definite for every $u\in TM\setminus \mathbf 0$.

In particular, if $F$ is a Finsler metric,  $F^2$ is $C^1$ (in fact it is $C^2$ if and only if it comes from a Riemannian metric \cite{warner65}) and $F$ satisfies the triangle inequality (see \cite[Theorem 1.2.2]{BaChSh00}).

One can naturally define the distance for any (continuous) pseudo-Finsler metric as
\begin{equation}\label{distance}
d_F (p,q)=\inf_{\gamma\in C(p,q)} \ell_F(\gamma),
\end{equation}
where $C(p,q)$ is the set of piecewise smooth curves from $p$ to $q$ and
$\ell_F(\gamma)$ is the $F$-length of $\gamma:[a,b]\rightarrow \R$, namely,
$$\ell_F(\gamma)=\int_a^b F(\dot\gamma(s))\df s.$$ Observe that
the distance associated to a (continuous) pseudo-Finsler metric is a quasi-metric. 
From now on we will use  ``almost isometry'' for maps $\varphi:M_1\rightarrow M_2$, such that  $\varphi:(M_1,d_{F_1})\rightarrow (M_2,d_{F_2})$ is an almost isometry.

Let us define an average Riemannian metric associated to every pseudo-Finsler metric. Given a point $p\in M$, let us denote $S=\{v\in T_pM: F(v)=1\}$, which is usually called the indicatrix of $F$, and $B=\{v\in T_pM: F(v)\leq 1\}$. Now let $\Omega$ be the unique volume form (or density) in $T_pM$ such that $B$ has volume one and let $\omega$ be the volume form  in the hypersurface $S$ of $T_pM$  that in $u\in S$ is given as
\[\omega(\eta_1,\eta_2,\ldots,\eta_n)=\Omega(u,\eta_1,\eta_2,\ldots,\eta_n)\]
for every $\eta_1,\eta_2,\ldots,\eta_n\in T_u S$. Then define
\[h(v,w)=\int_S g_u(v,w)\omega\]
for $v,w\in T_pM$ (see \cite[Section 2]{MRTZ09}  and also \cite{ricardo}). 
Observe that if $F$ is a convex Finsler metric, namely, a pseudo-Finsler metric satisfying the triangle inequality in the fibers, then $h$ is positive. Indeed, $F$ is convex if and only if $g_u$ is semi-definite positive for every $u\in TM\setminus 0$ (see for instance \cite[Theorem 2.14]{JaSan11}) and $g_u(u,u)=F(u)^2>0$.

Given a pseudo-Finsler metric $F$ in $M$, we can define a reversible pseudo-Finsler metric
$\hat{F}:TM\rightarrow \R$, which we will call the symmetrized pseudo-Finsler metric of $F$, as
\begin{equation}\label{symmetrized}
\hat{F}(v)=\tfrac 12\big[F(v)+F(-v)\big]
\end{equation}
for every $v\in TM$. Observe that if $F$ is convex (resp. strongly convex), then $\hat{F}$ is also convex (resp. strongly convex), see for instance \cite[Corollary 4.3]{JaSan11}. 

\begin{lemma}\label{thm:diffalmostisomFinsler}
Let $\varphi : (M_1, F_1) \rightarrow (M_2, F_2)$ be an almost isometry  of Finsler manifolds and $\hat{F}_1$ and $\hat{F}_2$ the symmetrized Finsler metrics of $F_1$ and $F_2$ respectively. Then $\varphi:(M_1, \hat{F}_1) \rightarrow (M_2, \hat{F}_2)$ is an isometry and $\varphi$ is smooth.
\end{lemma}

\begin{proof}
First observe that given $\alpha:[a,b]\subseteq \R\rightarrow M_1$,
\begin{equation}\label{hatF_1}
\ell_{\hat{F}_1}(\alpha)=\frac 12\int_a^b (F_1(\dot\alpha)+F_1(-\dot\alpha))ds=\frac 12 (\ell_{F_1}(\alpha)+\ell_{F_1}(\hat{\alpha})),
\end{equation}
where $\hat{\alpha}:[a,b]\rightarrow M$ is defined as $\hat{\alpha}(s)=\alpha(a+b-s)$ for $s\in [a,b]$.
Likewise, for $\beta:[a,b]\subseteq \R\rightarrow M_2$,
\begin{equation}\label{hatF_2}
\ell_{\hat{F}_2}(\beta)=\frac 12 (\ell_{F_2}(\beta)+\ell_{F_2}(\hat{\beta})).
\end{equation}
Consider now $\gamma:[a,b]\rightarrow M_1$ and the set of partitions of $[a,b]$:
\[{\mathcal P}=\{(t_0,t_1,\ldots,t_{n})\in\R^{n+1}:a=t_0<t_1<t_2<\ldots<t_{n-1}<t_{n}=b, n\in\N\}.\]
First observe that the quasi-metric defined by a Finsler metric makes the underlying space a length space, and then the length given by \eqref{d-length} coincides with the Finsler length. As $\varphi$ is an almost isometry, by Proposition \ref{defequiv}, there exists $f:M_2\rightarrow \R$, such that
\begin{multline}\label{F_1F_2}
\ell_{F_2}(\varphi\circ\gamma)=\sup_{\mathcal P}\sum_{i=0}^{n-1} d_2
(\varphi\circ\gamma(t_i),\varphi\circ\gamma(t_{i+1}))\\
=\sup_{\mathcal P}\sum_{i=0}^{n-1} d_1
(\gamma(t_i),\gamma(t_{i+1}))+f(\varphi\circ\gamma(b))-f(\varphi\circ\gamma(a))\\
=\ell_{F_1}(\gamma)+f(\varphi\circ\gamma(b))-f(\varphi\circ\gamma(a)).
\end{multline}
In particular,
\begin{align*}
\ell_{F_2}(\varphi\circ\alpha)=\ell_{F_1}(\alpha)+f(\varphi\circ\alpha(b))
-f(\varphi\circ\alpha(a))
\end{align*}
and
\begin{align*}
\ell_{F_2}(\varphi\circ\hat{\alpha})=\ell_{F_1}(\hat{\alpha})+f(\varphi\circ\alpha(a))
-f(\varphi\circ\alpha(b)).
\end{align*}
The above equations, together with \eqref{hatF_1} and \eqref{hatF_2}, imply
that $\ell_{\hat{F}_2}(\varphi\circ\alpha)=\ell_{\hat{F}_1}(\alpha)$ and then $\varphi$ is an isometry
 between the Finsler manifolds $(M_1, \hat{F}_1)$ and $(M_2, \hat{F}_2)$. Moreover, it is also an isometry between the average Riemannian metrics of $\hat{F}_1$  and $\hat{F}_2$.
 This implies that $\varphi$ is smooth (see \cite[Theorem 8]{MySt39}).
\end{proof}

Given a diffeomorphism $\varphi:M_1\rightarrow M_2$,  and a Finsler metric $F$ in $M_1$, let us denote by $\varphi_*(F)$ the Finsler metric in $M_2$ obtained as the push-forward of $F$ by $\varphi$.
\begin{prop}\label{diff_fisometries}
Let $(M_1,F_1)$ and $(M_2,F_2)$ be two Finsler manifolds. If there exists an almost isometry $\varphi:(M_1,F_1)\rightarrow (M_2,F_2)$, then there exists a smooth $f:M_2\rightarrow \R$ such that $\varphi_*(F_1)=F_2-\df f$. Conversely, if $\varphi_*(F_1)=F_2-\df f$, the map $\varphi$ is an almost isometry.
\end{prop}

\begin{proof}
By Lemma \ref{thm:diffalmostisomFinsler}, $\varphi$ is smooth. To see that $f$ is smooth, given a point $x\in M_2$, consider neighborhoods $U$ of $x$ and $V$ of $\varphi^{-1}(x)$, where the functions  $d_{F_2}(\,\cdot\,,x)$ and $d_{F_1}(\,\cdot\,,\varphi^{-1}(x))$ are smooth in $U\setminus \{x\}$ and $V\setminus \{\varphi^{-1}(x)\}$ respectively. Then choose $x_0\in U\cap \varphi(V)$, such that $x_0\not=x$. Therefore
\[f(y)=f(x_0)-d_1(\varphi^{-1}(x_0),\varphi^{-1}(y))+d_2(x_0,y)\]
for every $y\in M$ and the expression on the right side is smooth
for every $y\in U\cap \varphi(V)\setminus \{x_0\}$ and in particular for $x$.

To show the implication to the right, observe that \begin{equation}\label{dFlength}\ell_{F_2-\df f}(\alpha)=\ell_{F_2}(\alpha)+f(\alpha(a))-f(\alpha(b)),
\end{equation}
where $\alpha:[a,b]\rightarrow M_2$ is a piecewise smooth curve. Then the above equality  and \eqref{F_1F_2} imply that
\[\ell_{\varphi_*(F_1)}(\varphi\circ\gamma)=\ell_{F_1}(\gamma)=\ell_{F_2}(\varphi\circ\gamma)+f(\varphi\circ\gamma(a))-f(\varphi\circ\gamma(b))=
\ell_{F_2-\df f}(\varphi\circ\gamma),\]
for any curve $\gamma:[a,b]\subseteq\R\rightarrow M_1$.  Therefore $\varphi_*(F_1)$ and $F_2-\df f$ coincide.

The converse follows straightforward from \eqref{dFlength}.
\end{proof}

\begin{prop}\label{isoextended}
Let $(M,F)$ be a Finsler manifold. Then the extended isometry group $\widetilde{\rm Iso}(M,F)$ is a closed
subgroup of ${\rm Iso}(M,\hat{F})$, where $\hat{F}$ is given in \eqref{symmetrized}. In particular, $\widetilde{\rm Iso}(M,F)$ is a Lie group.
\end{prop}
\begin{proof}
Let us show first that $\widetilde{\rm Iso}(M,F)\subset{\rm Iso}(M,\hat{F})$.
Let $\varphi$ be an almost isometry. Then by Proposition \ref{diff_fisometries}
we know that
 $\varphi$ is an isometry of  $(M,\hat{F})$. Moreover, we know that ${\rm Iso}(M,\hat{F})$
is a Lie group since it is a closed subgroup of ${\rm Iso}(M,h)$, where
$h$ is the average Riemannian metric of $\hat{F}$. The closedness of $\widetilde{\rm Iso}(M,F)$ in ${\rm Iso}(M,\hat{F})$ follows from the continuity of the triangular function.
\end{proof}
Standard examples of non-reversible Finsler metrics are the so-called
Randers metrics. Precisely, a Randers  metric $R$ on a manifold $M$ is a metric defined as
\begin{equation}\label{randers}
R(v)=\sqrt{h(v,v)}+\omega(v),
\end{equation}
for  $v\in TM$ with $h$  a Riemannian metric and $\omega$ a one-form with $h$-norm less than one at every point $x\in M$.
\begin{cor}\label{hisometry}
Let $(M,R)$ be a Randers manifold with $R$ as in \eqref{randers} and $\varphi:M\rightarrow M$ an almost isometry for $R$. Then $\varphi$ is an isometry for $h$.
\end{cor}
\begin{proof}
Just observe that the symmetrized Finsler metric of $R$ is given by $\hat{R}(v)=\sqrt{h(v,v)}$
for $v\in TM$. The result follows then from Proposition~\ref{isoextended}.
\end{proof}
\begin{rem} 
Let us observe that almost isometries of Finsler metrics are projective transformations, since they preserve geodesics up to parametrization (see Corollary \ref{pregeo}). In particular, they are related with the class of special projective transformation (see \cite[Theorem 1]{Raf13}). Moreover, in \cite{Mat12}, the author classifies Finsler metrics with extended isometry group of dimension greater than Wang's number $n(n-1)/2+1$.
\end{rem}

\section{Conformal maps in distinguishing conformastationary spacetimes}\label{stationary}
 For the basic notions on Lorentzian geometry we refer the reader to \cite{BEE96,oneill83}.
 Let $(M,g)$ be a distinguishing spacetime.
 We say that a vector field $K$ in $M$ is conformal if $\mathcal L_Kg=\lambda g$ for some smooth function $\lambda: M\rightarrow \R$, where $\mathcal L$ is the Lie derivative.
Let us denote by ${\rm CTF}(M,g)$ the subset of complete timelike conformal  vector fields of $(M,g)$, which is a subset of the vector space of conformal fields of $(M,g)$,
denoted  by ${\rm CF}(M,g)$. Observe that for any $K\in {\rm CTF}(M,g)$, since
$(M,g)$ is distinguishing, the main theorem in \cite{JaSan08} allows one to express $(M,g)$ as a standard conformastationary spacetime with respect to $K$, namely, a splitting $(S\times\R,g^{K})$ such that $g^K$ is expressed as
\begin{equation}\label{e1} g^K((v,\tau),(v,\tau))=\Omega^K( g^K_0(v,v)+2 \omega^K(v) \tau-\tau^2),
\end{equation}
for $(v,\tau)\in TS\times \R$, where $\Omega^K$ is a positive smooth real function
in $S\times \R$ and $g^K_0$ and $\omega^K$ are respectively  a Riemannian metric and a one-form on $S$. Here, the timelike conformal vector field $K$ is $\partial_t$ (the partial in the second variable). When ${\rm CTF}(M,g)$ is not empty,  we say that $(M,g)$ is {\it conformastationary} and given $K\in {\rm CTF}(M,g)$, $(S\times\R,g^K)$ is a {\it $K$-conformastationary decomposition} or $K$-splitting for short.
  Assume that the spacetime $(M,g)$ is time-oriented by $K$.
By a well-known Fermat's principle for conformastationary manifolds, (see for instance the review \cite{Jav12a} and references therein), it turns out that the (future-pointing) lightlike geodesics of the conformastationary spacetime have a natural
correspondence with the geodesics of a Finsler metric $F^K$ of Randers type in $S$, defined by
\begin{equation}\label{fermatmetric}
F^K(v)=\sqrt{g^K_0(v,v)+\omega^K(v)^2}+\,\omega^K(v),
\end{equation}
for any $v\in TS$, which we call the \emph{Fermat metric} associated to the conformastationary metric \eqref{e1}.
More precisely, the projections in $S$ of future-pointing lightlike geodesics \[[0,1]\ni s\longmapsto (x(s),t(s))\in S\times \R\] are Randers pregeodesics in $(S,F^K)$, and the length of the Randers geodesic equals $t(1)-t(0)$. Moreover, the causality of the conformastationary spacetime can be characterized in terms of Hopf-Rinow properties of the Randers metric (see \cite[Theorem 4.3]{CJS11}).
\smallskip

 Given $K\in {\rm CTF}(M,g)$, there are several $K$-conformastationary decompositions of $(M,g)$, but all the Fermat metrics of them are in the same class as we show below.
\begin{defi}
Given two Finsler metrics $F_1$ and $F_2$ in a manifold $S$, we say
that they are projectively related if there exists a smooth function $f:S\rightarrow \R$ such that $F_1=F_2+{\rm d}f$. In this case, we will say that they belong to the same class or in notation $[F_1]=[F_2]$.
\end{defi}
Observe that if $F_1$ is a Randers metric, then all the Finsler metrics in the same class are also Randers. Let us denote by $\mathrm{RC}(S)$ the subset of classes of Randers metrics in the manifold $S$.  
\begin{lemma}\label{Fermatclass}
Given  $K\in\mathrm{CTF}(M,g)$,
all the Fermat metrics associated to the $K$-splittings of $(M,g)$ belong to the same class, which will be denoted by $i(K)$. Moreover, there is a one-to-one correspondence between $K$-splittings of $(M,g)$ and representatives of the class $i(K)\in \mathrm{RC}(S)$.
\end{lemma}
\begin{proof}
If we fix a $K$-splitting $(S\times \R,g^K)$, then any other $K$-splitting is determined by a spacelike section of $S\times\R$, which always can be expressed as
\[S^f=\{(x,f(x))\in S\times\R :x\in S\},\]
where $f:S\rightarrow\R$ is a smooth function. If $F^K$ is the Fermat metric associated to $(S\times \R,g^K)$, it is easy to see that $S^f$ is spacelike if and only if $F-{\rm d}f$ is positive definite (see \cite[Proposition 5.8]{CJS11}) and in this case the Fermat metric associated to the $K$-splitting constructed from the slice $S^f$  is $F-{\rm d}f$.
\end{proof}
Let us observe that if a map $\varphi:(S,F_1)\rightarrow (S,F_2)$ is an almost isometry,  it remains almost isometry when we replace $F_1$ or $F_2$ with a Finsler metric in the same class. Then the subset 
\[\widetilde{\rm Iso}(S,[F_1],[F_2])=\{\varphi:(S,F_1)\rightarrow (S,F_2): \text{$\varphi$ is an almost isometry}\}\]
is well-defined. Furthermore, fixing an almost isometry $\varphi_0:(S,F_1)\rightarrow (S,F_2)$, we get the identifications $\widetilde{\rm Iso}(S,[F_1],[F_2])\cong
\widetilde{\rm Iso}(S,F_1)\cong \widetilde{\rm Iso}(S,F_2)$, and we can think in $\widetilde{\rm Iso}(S,[F_1],[F_2])$ as a manifold.

Let us see  that we can characterize all the conformal maps of a conformastationary spacetime in terms of almost isometries of Fermat metrics.
\begin{thm}\label{fundamental}
Let $\psi:(M,g)\rightarrow (M,g)$ be a conformal map of a conformastationary spacetime $(M,g)$. Let $K\in {\rm CTF}(M,g)$ and denote $W=\psi_*(K)$. Then 
using arbitrary conformastationary decompositions of $K$ and $W$, the map
\[\psi:(S\times\R,g^K)\rightarrow (S\times \R,g^W)\]
 is given by
$\psi(x,t)=(\varphi(x),t+f(x))$, where $\varphi:(S,F^K)\rightarrow (S,F^W)$ is an almost isometry.
\end{thm}

\begin{proof}
First of all, observe that $\psi_*(K)\in {\rm CTF}(M,g)$, because conformal maps preserve causality and a diffeomorphism preserves completeness of vector fields. Then there exist respectively $K$ and $W$-conformastationary decompositions of $(M,g)$ (see the introduction of this section). It is straightforward that $\psi:(S\times\R,g^K)\rightarrow (S\times\R,g^W)$ is given by
$\psi(x,t)=(\varphi(x),t+f(x))$ for certain smooth functions $\varphi:S\rightarrow S$
and $f:S\rightarrow\R$.

In order to see that $\varphi$ is an almost isometry, observe that conformal maps preserve lightlike pregeodesics (see for instance \cite[Lemma 9.17]{BEE96}) and as the projection of lightlike geodesics are Fermat pregeodesics (see \cite[Theorem 4.5]{CJM11}), it follows that $\varphi$ maps Fermat geodesics of $(S,F^K)$ into Fermat pregeodesics of $(S,F^W)$. Moreover,  a geodesic $\gamma:[0,1]\rightarrow S$ of the Fermat metric $F^K$ satisfies that \[\ell_{F^W}(\varphi\circ \gamma)=\ell_{F^K}(\gamma)+f\big(\varphi(\gamma(1))\big)-f\big(\varphi(\gamma(0))\big)\] (use again that the length of the Fermat geodesic is given by the difference of the time component at the endpoints, see \cite[Theorem 4.5]{CJM11}). The above equation together with \eqref{dFlength} implies that
\[\ell_{\varphi_*(F^K)}(\varphi\circ\gamma)=\ell_{F^K}(\gamma)=\ell_{F^W-{\rm d}f}(\varphi\circ\gamma)\]
for any geodesic $\gamma$ of $(S,F^K)$, and then that  $\varphi_*(F^K)(v)=F^W(v)-{\rm d} f (v)$ for every $v\in TS$. Proposition \ref{diff_fisometries} yields that
$\varphi$ is an almost isometry.
\end{proof}

This result allows us to describe the conformal group of the conformastationary spacetime as we will see later.

\subsection{$K$-conformal maps}

 In this subsection we will fix $K\in {\rm CTF}(M,g)$ and we will study a relevant class of transformations, namely, those that preserve the conformal observer determined by $K$. We will also fix a $K$-conformastationary decomposition $(S\times\R,g^K)$ and we will remove the $K's$ in the expressions of \eqref{e1} and \eqref{fermatmetric}.
\begin{defi}\label{def:statdiffeos}
We say that a diffeomorphism $\psi:(M,g)\rightarrow (M,g)$ is
{\em $K$-stationary}  if it
  preserves the conformal vector field $K$, namely, $\psi_*(K)=K$. Moreover, we say that it is {\it $K$-conformal} if it is $K$-stationary and conformal.
\end{defi}
Observe that when $g(K,K)=-1$, that is, the conformastationary spacetime is a normalized stantard stationary spacetime, 
or more generally when $g(K,K)$ is constant,
then every $K$-conformal map is an isometry of $g$. 
\begin{cor}\label{conformal_against_Fermat}
Let $\psi:S\times \R\rightarrow S\times \R$ be a $K$-conformal map of the conformastationary spacetime. Then there exist functions $\varphi:S\rightarrow S$ and
$f:S\rightarrow \R$ such that $\psi(x,t)=\big(\varphi(x),t+f(x)\big)$ and $\varphi$ is an almost isometry for the Fermat metric $F$ of $(S\times \R,g)$ with $f$ satisfying \eqref{eq:defquasiisom} for $d_F$. Moreover, $\varphi$ is a Riemannian isometry for the metric in $S$ given by
\begin{equation}\label{metrica-h}
h(v,v)=g_0(v,v)+\omega(v)^2
\end{equation}
for $v\in TS$.
\end{cor}
\begin{proof}
It follows immediately from Theorem \ref{fundamental} and Proposition \ref{hisometry}.
\end{proof}
Let us denote by ${\rm Conf}_K(M,g)$ the subset of the Lie group of  conformal maps ${\rm Conf}(M,g)$ consisting of $K$-conformal maps.
\begin{lemma}
 ${\rm Conf}_K(M,g)$ is a $\mathcal C^1$-closed subgroup of ${\rm Conf}(M,g)$;
 in $\mathrm{Conf}_K(M,g)$, all the $\mathcal C^k$-topologies coincide\footnote{It is not clear to the authors whether the same statement,
 i.e., coincidence of all $\mathcal C^k$-topologies, holds in the whole set $\mathrm{Conf}(M,g)$.} for $k=0,1,\ldots,+\infty$.
Moreover, the one-parameter subgroup ${\mathcal K}$ generated by the flow of $K$ is closed and normal in ${\rm Conf}_{K}(M,g)$.
\end{lemma}
\begin{proof}
Let us show that all $C^k$-topologies coincide in $\mathrm{Conf}_K(M,g)$. Let $g_R$ be the Riemannian metric in
$M$ given as
\[g_R(v,w)=g(v,w)-2\frac{g(v,K)g(w,K)}{g(K,K)}\]
for any $v,w\in TM$. Then $\mathrm{Conf}_K(M,g)\subseteq\mathrm{Conf}(M,g_R)$. Observe that
$\mathrm{Conf}(M,g_R)$ is contained in the isometry group of a Riemannian metric $h$ or it is the conformal group of the round sphere or
of the Euclidean space $\R^n$ (see \cite{Fer96}). In the three cases, the assertion about the equivalence of the $\mathcal C^k$-topologies, for $k=0,1,\ldots,\infty$ is true and it is inherited by $\mathrm{Conf}_K(M,g)$.

Closedness in the $\mathcal C^1$-topology is obvious from the very definition, since the identity $\psi_*(K)=K$
involves only the first differential of $\psi$, see Definition~\ref{def:statdiffeos}.
Let $\psi \in \mathrm{Conf}_K(M, g)$ be given by $\psi(x,t) = (\phi(x),t + f(x))$ for some smooth
 $f:M\rightarrow\R$ and $\phi \in \widetilde{\rm Iso}(S,F)$. Then, its inverse is given by
$$\psi^{-1}(x,t) = \big(\phi^{-1}(x),t - f(\phi^{-1}(x))\big)$$
Now, let $K^{T} : M \rightarrow M$ be the time $T$  of the  flow of $K$, i.e., $K^{T}(x,t) = (x,t + T)$. It is easy to see that $\psi \circ K^{T} \circ \psi^{-1} = K^{T}$. Thus $\mathcal{K}$ is normal in $\mathrm{Conf}_K(M, g)$.

Let us finally show that $\mathcal{K}$ is closed. Take a sequence $\{\varphi_n\}_{n\in \N}$ in $\mathcal{K}$ converging to $\varphi\in \mathrm{Conf}_K(M, g)$ such that $\varphi_n=K^{T_n}$. Then the sequence $\{T_n\}_{n\in\N}$ must be bounded, otherwise $\varphi_n$ would not have pointwise convergence.  Therefore there exists a subsequence such that $T_{n_k}\rightarrow T\in \R$ and $\varphi=K^T$.
\end{proof}

\begin{prop}\label{bijection}
Let $M=S\times \R$ be endowed with the conformastationary Lorentz metric \eqref{e1}, set $K=\partial_t$ and let $F$ denote the Fermat metric on $S$
given in \eqref{fermatmetric}.  \hfill\break
The map $\pi: {\rm Conf}_K(M,g)\rightarrow \widetilde{\rm Iso}(S,F)$ defined as
$\pi(\psi)=\varphi$ (see Corollary~\ref{conformal_against_Fermat}) is a Lie group homomorphism. Moreover, $\pi$ can be projected to the quotient
\[\bar{\pi}: {\rm Conf}_K(M,g)/{\mathcal K}\rightarrow \widetilde{\rm Iso}(S,F)\]
 and gives an isomorphism of Lie groups.
\end{prop}
\begin{proof}
The only thing that does not follow immediately from Corollary~\ref{conformal_against_Fermat} is that $\bar{\pi}$ is one-to-one,
i.e., that $\mathcal K$ is the kernel of the map $\pi$ and $\pi$ is surjective. Let us show that $\bar\pi$ is injective. Assume that there exist two $K$-conformal maps $\psi_1$ and $\psi_2$ projecting to the same almost isometry $\varphi$. Then, by Proposition~\ref{diff_fisometries}, there exists a smooth $f:S\rightarrow \R$ such that $\varphi_*(F)=F-{\rm d} f$ and, from Corollary~\ref{conformal_against_Fermat},  $\psi_1(x,t)=(\varphi(x),t+f(x)+c_1)$ and $\psi_2(x,t)=(\varphi(x),t+f(x)+c_2)$. Therefore $\psi_2^{-1}\circ\psi_1\in \mathcal K$ and $[\psi_1]=[\psi_2]$. Finally, we will see that it is surjective. Given an almost isometry $\varphi$, choose a function $f:S\rightarrow \R$ as in Proposition \ref{defequiv} and construct the map $\psi:S\times \R\rightarrow S\times \R$ as $\psi(x,t)=(\varphi(x),t+f(x))$, which preserves $K$. By the Fermat principle \cite[Theorem 4.5]{CJM11}, $\psi$ maps lightlike pregeodesics to lightlike pregeodesics. This means that it preserves the lightlike cone and then it is conformal (see \cite{DaNo80} or \cite{Kul79}).

Finally observe that as $\pi$ is a continuous homomorphism of Lie groups, it is differentiable and the same thing happens with $\bar{\pi}$ (see for instance \cite[Corollary 1.10.9 and Proposition 1.11.8]{DuKo00}).
\end{proof}
\begin{cor}\label{genericity}
Given a manifold $S$, for a \emph{generic} (see Remark~\ref{rem:generic} below)
set of data $(g_0,\omega)$, the conformastationary metric $g=g(g_0,\omega)$
given in \eqref{e1}
on $M=S\times \R$ has discrete $K$-conformal group $\mathrm{Conf}_K(M,g)/{\mathcal K}$.
\end{cor}
\begin{proof}
Given $(g_0,\omega)$, denote by $h$ the Riemannian metric on $S$ given in \eqref{metrica-h}.
By Proposition~\ref{bijection}, $\mathrm{Conf}_K(M,g)/{\mathcal K}\cong\widetilde{\rm Iso}(S,F)$, with $F=\sqrt h+\omega$, see \eqref{fermatmetric}.
By Corollary~\ref{hisometry}, $\widetilde{\rm Iso}(S,F)\subset\mathrm{Iso}(S,h)$, thus $\mathrm{Conf}_K(M,g)/{\mathcal K}$ is discrete if the Riemannian
isometry group $\mathrm{Iso}(S,h)$ is discrete. It is well known, see for instance
\cite{Uhlen72} for the compact case, that, for a $\mathcal C^{2,\alpha}$-generic set of Riemannian metrics $h$, the isometry group $\mathrm{Iso}(S,h)$ is discrete.
On the other hand, the map $(g_0,\omega)\mapsto h= g_0+\omega\otimes\omega$ is clearly continuous
(in the $\mathcal C^0$-topology, for instance), open and surjective. In particular, the inverse image of a dense $G_\delta$ is a dense $G_\delta$.
Hence, $\mathrm{Conf}_K(M,g)/{\mathcal K}$ is discrete for a generic set of pairs $(g_0,\omega)$.
\end{proof}
\begin{rem}\label{rem:generic}
Recall that a subset of a metric space is \emph{generic} if it contains
a dense $G_\delta$, i.e., a countable intersection of dense open subsets. Here, genericity is meant in the space
$\mathrm{Riem}^{2,\alpha}(S)\times\Gamma^{2,\alpha}(TS^*)$, where
$\mathrm{Riem}^{2,\alpha}(S)$ is the set of Riemannian metric tensors of class $\mathcal C^{2,\alpha}$ on $S$, and
$\Gamma^{2,\alpha}(TS^*)$ is the space of $1$-differential forms on $S$ of class $\mathcal C^{2,\alpha}$.
For the sake of precision, when $S$ is not compact, suitable asymptotic assumptions have to be taken into consideration for the correct definition of these spaces, however
we will not get into details of this type of technicalities here.
\end{rem}
\begin{cor}\label{compactLiegroup}
If $S$ is compact, then ${\rm Conf}_K(S\times \R,g)/{\mathcal K}$ and  $\widetilde{\rm Iso}(S,F)$
are compact Lie groups.
\end{cor}
\begin{proof}
Observe that $\widetilde{\rm Iso}(S,F)$ is a closed subgroup of the compact Lie group
${\rm Iso}(S,h)$ (see Proposition
\ref{isoextended} and Corollary \ref{hisometry}) and then it is compact. Proposition \ref{bijection}
gives the compactness of ${\rm Conf}_K(S\times \R,g)/{\mathcal K}$.
\end{proof}

\subsection{Characterization of conformal maps}
 Given  $K\in {\rm CTF}(M,g) $, denote by ${\rm CTF}_K(M,g)$ the subset of ${\rm CTF}(M,g)$ consisting of all timelike conformal fields $W$ for which there exists a conformal map
$\psi:(M,g)\rightarrow (M,g)$ satisfying $\psi_*(W)=K$.
\begin{prop}
With the above notation 
\[\mathrm{CTF}_K(M,g)\cong \mathrm{Conf}(M,g)/\mathrm{Conf}_K(M,g).\]
\end{prop}
\begin{proof}
Observe that the group $\mathrm{Conf}(M,g)$ acts on $\mathrm{CF}(M,g)$ by push-forward: $(\psi,W)\mapsto\psi_*(W)$.
Given $K\in\mathrm{CTF}(M,g)$, the \emph{stabilizer} of $K$ is the subgroup $\mathrm{Conf}_K(M,g)$, and
the set $\mathrm{CTF}_K(M,g)$ is precisely the \emph{orbit} of $K$ by this action and then a submanifold of the vector space $\mathrm{CF}(M,g)$. Therefore, one has an
identification of $\mathrm{CTF}_K(M,g)$ with the quotient manifold  $\mathrm{Conf}(M,g)/\mathrm{Conf}_K(M,g)$.
\end{proof}
\begin{thm}\label{completeConf}
With the above notation, the map 
\[ \pi:\mathrm{Conf}(M,g)/{\mathcal K}\rightarrow 
\mathrm{Conf}(M,g)/\mathrm{Conf}_K(M,g)\cong \mathrm{CTF}_K(M,g),\]
defined in the natural way, is a submersion with fibers diffeomorphic to $\widetilde{\rm Iso}(S,F)$. If $\widetilde{Iso}(S,F)$ is discrete, then
$\pi$ is a covering map.
\end{thm}
\begin{proof}
It is easy to see that $\pi$ is well-defined and smooth by the universal property of quotient maps. It is also immediate to see that it is a submersion and the fibers are diffeomorphic to $\mathrm{Conf}_K(M,g)/{\mathcal K}$, which, by Proposition \ref{bijection}, is diffeomorphic to $\widetilde{\rm Iso}(S,F)$. For the last claim, observe that $\widetilde{\rm Iso}(S,F)\cong  {\rm Conf}_K(M,g)/{\mathcal K}$ acts in $\mathrm{Conf}(M,g)/{\mathcal K}$ by composition to the right. If $\widetilde{\rm Iso}(S,F)$ is discrete then $\pi$ is a local diffeomorphism and as a consequence the action of  $\widetilde{\rm Iso}(S,F)$ is properly discontinuous, which implies that $\pi$ is a covering.
\end{proof}
\begin{cor}\label{compactconf}
Assume that $S$ is compact. Then $\mathrm{Conf}(M,g)/{\mathcal K}$ is compact if and only if $\mathrm{CTF}_K(M,g)$ is compact.
\end{cor}
\begin{proof}
It follows from Theorem \ref{completeConf} taking into account that if $S$ is compact, then $\widetilde{\rm Iso}(S,F)$ is also compact and  the total space of a submersion with compact fibers is compact if and only if the base is compact.
\end{proof}
 Let us define 
\[\mathrm{Conf}_{K,W}(M,g)=\{\psi\in \mathrm{Conf}(M,g): \psi_*(W)=K\},\]
which can be identified with $\mathrm{Conf}_{K}(M,g)$ by fixing one element of $\mathrm{Conf}_{K,W}(M,g)$. 
\begin{thm}\label{finalTh}
Choose $F_1\in i(K)$ and $F_2\in i(W)$. Then every map $\psi\in\mathrm{Conf}_{K,W}(M,g)$ can be expressed as $\psi(x,t)=(\varphi(x),t+f(x))$, where $\varphi:(S,F_1)\rightarrow (S,F_2)$ is an almost isometry and $f:S\rightarrow\R$ satisfies \eqref{eq:defquasiisom}. Moreover, there is a smooth surjective map $\pi:\mathrm{Conf}_{K,W}(M,g)\rightarrow \widetilde{\rm Iso}(S,i(K),i(W))$ defined as $\pi(\psi)=\varphi$ and if $\pi(\psi_1)=\pi(\psi_2)$, then
$\psi_1=\psi_2\circ T_c$, for some $c\in\R$, where $T_c\in {\mathcal K}$.
\end{thm}
\begin{proof}
The proof follows easily using Theorem \ref{fundamental} taking into account Lemma \ref{Fermatclass} and arguing as
in Proposition~\ref{bijection}. 
\end{proof}

\section*{Acknowledgments}

 The authors warmly acknowledge  Professors E. Garc\'ia-R\'io, V. Matveev
and M. S\'anchez for helpful conversations on the topics of this
paper.

\end{document}